\documentclass[reqno,12pt]{amsart}

\usepackage{amsmath, amsthm, amsfonts, amssymb}
\input{mathrsfs.sty}

\usepackage{amsmath}
\usepackage{amsfonts}
\usepackage{amssymb}
\usepackage{eufrak}
\usepackage{amscd}
\usepackage{amsthm}
\usepackage{amstext}
\usepackage[all]{xy}

\newcommand{\Ran}{\operatorname{Ran}}
\newcommand{\dist}{\operatorname{dist}}

   \theoremstyle{plain}
   \newtheorem{thm}{Theorem}[section]
   \newtheorem{prop}[thm]{Proposition}
   \newtheorem{lem}[thm]{Lemma}
   
   \theoremstyle{definition}

   \newtheorem{example}{Example}
	\newtheorem*{question*}{Question}

   \theoremstyle{remark}

\author{V. Manuilov}

\date{}

\address{Moscow Center for Fundamental and Applied Mathematics, Moscow State University,
Leninskie Gory 1, Moscow, 
119991, Russia}

\email{manuilov@mech.math.msu.su}

\thanks{The author acknowledges support by the RNF grant 24-11-00124.}

\title{On large submodules in Hilbert $C^*$-modules}

\begin{document}

\begin{abstract}
We consider several natural ways of expressing the idea that a one-sided ideal in a $C^*$-algebra (or a submodule in a Hilbert $C^*$-module) is large, and show that they differ, unlike the case of two-sided ideals in $C^*$-algebras. We then show how these different notions, for ideals and for submodules, are related. We also study some permanence properties for these notions. Finally, we use essential right ideals to extend the inner product on a Hilbert $C^*$-module to a part of the dual module.

\end{abstract}

\maketitle

\section{Introduction}


For a closed right ideal $J$ in a $C^*$-algebra $A$ there are at least two different ways to express the idea that $J\subset A$ is large. 

$J$ is {\it essential} if
\begin{itemize}
\item[$(i1)$]
$J\cap I\neq 0$ for any non-zero right ideal $I\subset A$. 
\end{itemize}

This may seem too algebraic and not suited for ideals in $C^*$-algebras as it uses general (non-closed) ideals, but the following equivalent property (see e.g. \cite{Lam}) has no such algebraic flavor:
\begin{itemize}
\item[$(i1')$]
for each non-zero $a\in A$ there exists $b\in A$ such that $ab\in J$ and $ab\neq 0$.  
\end{itemize}


Using only closed submodules, one may define a similar notion:  

a closed right ideal $J$ is {\it topologically essential} if
\begin{itemize}
\item[$(i2)$]
$J\cap I\neq 0$ for any non-zero closed right ideal $I\subset A$.  
\end{itemize}

Clearly, $(i1)\Rightarrow(i2)$. Here we do not address the question, whether $(i1)$ and $(i2)$ are equivalent.

\smallskip

Another way to express largeness of a closed right ideal $J$ is:

$J$ is {\it thick} if
\begin{itemize} 
\item[$(i3)$]
$aJ=0$ implies $a=0$ for $a\in A$. 
\end{itemize} 

As far as we know, there is no standard name for ideals satisfying $(i3)$. A similar property was considered for Hilbert $C^*$-modules over $C^*$-algebras in \cite{M-MN}, where such submodules were called {\it thick} submodules. We keep here this terminology for ideals as well. 

Pay attention to the order of multiplication in $(i3)$. This property may be replaced by 
\begin{itemize} 
\item[$(i4)$]
$Ja=0$ implies $a=0$ for $a\in A$. 
\end{itemize}
The property $(i4)$ differs from $(i3)$ and is studied in \cite{Kaneda-Paulsen}, see Prop. 4.2. Our choice of the order of multiplication comes from the Hilbert $C^*$-module theory: if we consider a right ideal $J$ as a Hilbert $C^*$-module over $A$ with the $A$-valued inner product $\langle x,y\rangle=x^*y$, $x,y\in J$, then $aJ=0$ is equivalent to the condition that the orthogonal complement of $J$ in $A$ is zero, while $Ja=0$ has no interpretation in terms of the inner product.

The properties $(i1)$ and $(i3)$ can be considered also for {\it two-sided} ideals, and for them $(i1)$ and $(i3)$ are known to be equivalent (see, e.g., the comment preceding Theorem 3.1.8 in \cite{Murphy}). This equivalence is used in the theory of essential extensions of $C^*$-algebras. It is a common belief that, for one-sided ideals in $C^*$-algebras, $(i1)$ and $(i3)$ are not equivalent, but to our best knowledge there was no example showing the difference between these two properties.

Our first aim is to show that $(i1)$ (and $(i2)$) implies $(i3)$, but not vice versae. The implication $(i1)\Rightarrow(i3)$ is easy and known to specialists, and we recall it in Proposition \ref{Implication}. An example showing that $(i3)$ does not imply $(i1)$ is more elaborate. We provide it in the next section. This example is inspired by a similar Example 4 in \cite{Akemann-Bice}, and by \cite{Kaad-Skeide}.

Our next aim is to consider, for submodules of Hilbert $C^*$-modules, the conditions similar to the above conditions of largeness for ideals. We define properties $(m1)$--$(m5)$ similar to $(i1)$--$(i3)$ for submodules and show how some of them are related. In particular, we introduce the condition $(m5)$ as the uniqueness for extending bounded functionals from a submodule to the larger module, and show that $(m5)$ sits strictly between essentiality $(m1)$ and thickness $(m3)$. This sheds a new light on the examples constructed in \cite{Kaad-Skeide}, and later in \cite{M-MN}. A modification of our example shows that thickness behaves not so nicely as essentiality: the intersection of two thick submodules need not be thick, and for submodules $K\subset M\subset N$, $K$ need not be thick in $N$ when it is thick in $M$ and $M$ is thick in $N$.

We also compare some of the properties for a submodule $M$ in a Hilbert $C^*$-module $N$ with their counterparts for a certain ideal $J_M$ in the $C^*$-algebra $\mathbb K(N)$ of compact operators on $N$, related to $M$, and show that $(m1)$ (resp., $(m3)$) for $M$ in $N$ is equivalent to $(i1)$ (resp., $(i3)$) for $J_M$ in $\mathbb K(N)$.

Finally, we use essential right ideals to the problem of extending the inner product on a Hilbert $C^*$-module to a larger module. It is known that, for a Hilbert $C^*$-module $M$ over a $C^*$-algebra $A$, the dual module $M'$ need not be a Hilbert $C^*$-module, i.e. the $A$-valued inner product on $M$ does not extend to an inner product on $M'$. It is also known that it extends to $M'$ for any Hilbert $C^*$-module over any $W^*$-algebra \cite{Paschke}, and it extends to $M'$ for any Hilbert $C^*$-module over $A$ if and only if $A$ is monotone complete \cite{Frank}. On the other hand, there is an inclusion $M\subset M''\subset M'$, i.e. the second dual module $M''$ sits between $M$ and $M'$, and it was shown in \cite{Paschke} that the inner product on $M$ naturally extends to that on $M''$ for any $C^*$-algebra $A$. This may be useful, but if $M$ is reflexive, i.e. if $M''=M$, this is of no interest. Another case when a partial extension of the inner product is possible is the case when $M$ is an ideal in a $C^*$-algebra $A$, then, by trivial reason, the inner product extends from $M$ to $A$, but, generally, not further to $M'$, and if $M$ is essential then this extension is unique. In these cases, the extended inner product induces the norm equal to the norm in $M'$, and we are interested only in such extensions. For a general $C^*$-algebra $A$ there is no hope to extend the inner product to $M'$, and the question arises:

\begin{question*}
How far in $M'$ can the inner product be extended from $M$?

\end{question*}

We give a partial answer to this question by using essential right ideals of $A$ to construct a Hilbert $C^*$-submodule of $M'$, whereto the inner product can be extended.

Recall basic definitions related to Hilbert $C^*$-modules. For a $C^*$-algebra $A$, let $M$ be a right $A$-module with a compatible structure of a linear space, equipped with a sesquilinear map $\langle\cdot,\cdot\rangle:M\times M\to A$ such that
\begin{itemize}
\item
$\langle n,ma\rangle=\langle n,m\rangle a$ for any $n,m\in M$, $a\in A$;
\item
$\langle n,m\rangle=\langle m,n\rangle^*$ for any $n,m\in M$;
\item
$\langle m,m\rangle$ is positive for any $m\in M$, and if $\langle m,m\rangle=0$ then $m=0$. 
\end{itemize} 
Then $M$ is a pre-Hilbert $C^*$-module. If it is complete with respect to the norm given by $\|m\|^2=\|\langle m,m\rangle\|$ then $M$ is a Hilbert $C^*$-module.

A bounded anti-linear map $f:M\to A$ is a functional on $M$ if it is anti-$A$-linear, i.e. if $f(ma)=a^*f(m)$ for any $a\in A$, $m\in M$. The set of all functionals forms a Banach space which is also a right $A$-module with the action of $A$ given by $(fa)(m)=f(m)a$, $a\in A$, $m\in M$. The map $m\mapsto \langle \cdot,m\rangle$, $m\in M$, defines an isometric inclusion $j:M\subset M'$, but the inner product generally does not extend to $M'$, so the latter need not be a Hilbert $C^*$-module. Iterating, one gets the second dual module $M''$, which turns out to be a Hilbert $C^*$-module \cite{Paschke}.  More details on Hilbert $C^*$-modules can be found in e.g. \cite{Lance}, \cite{MT}.

\section{$(i2)$ is strictly stronger than $(i3)$}

\begin{prop}\label{Implication}
If a closed right ideal $J$ in a $C^*$-algebra $A$ is topologically essential then it is thick, i.e.  $(i2)\Rightarrow(i3)$.
\end{prop}
\begin{proof}
Let $J\subset A$.
Suppose the contrary: $J\subset A$ is essential, but $aJ=0$ does not imply $a=0$. Then there exists a non-zero $a\in A$ such that $a^*x=0$ for any $x\in J$. Let $I=\overline{aA}$ be the closure of the right ideal of all elements of the form $ab$, $b\in A$. As $J$ is topologically essential, $J\cap I\neq 0$, i.e. there exists $x\in J$ and a family $b_n\in A$, $n\in\mathbb N$, such that $x\neq 0$ and $x=\lim_{n\to\infty}ab_n$. Then, as $(a^*x)^*=x^*a=0$, we have
$$
x^*x=\lim_{n\to\infty}x^*ab_n=\lim_{n\to\infty}(x^*a)b_n=0, 
$$
hence $x=0$. The contradiction proves the claim. 
\end{proof}

Now we pass to the example showing that $(i3)$ does not imply $(i1)$, and even $(i2)$.

\smallskip

Let $X=[0,1]$, and let $\{t_i\}_{i\in\mathbb N}$ be a countable dense subset (e.g. the rational points). Let $\mathbb K$ denote the $C^*$-algebra of compact operators on a separable Hilbert space $H$, equipped with a fixed orthonormal basis $\{e_i\}_{i\in\mathbb N}$. Let $A=C(X,\mathbb K)$ be the $C^*$-algebra of continuous $\mathbb K$-valued functions on $X$.

Let $f:X\to \mathbb K$ be the function defined by 
$$
f(t)=\sum\nolimits_{i\in\mathbb N}\frac{1}{2^i}\chi_{[0,t_i]}(t)p_i,\quad t\in[0,1], 
$$
where $\chi_E$ denotes the characteristic function of the subset $E\subset X$ and $p_i$ denotes the rank one projection onto the vector $e_i$. Note that $f$ is continuous at irrational points and is not continuous at rational points. Set 
$$
J=\{x\in A: fx\in A\}, 
$$
i.e. $J$ consists of $\mathbb K$-valued functions $x$ that make the product $fx$ continuous. Clearly, $J$ is a right ideal in $A$ and is closed.

This ideal can be described in more detail. After we fix the basis $\{e_i\}$, a compact operator $x$ can be viewed as an infinite matrix, $x=(x_{ij})_{i,j\in\mathbb N}$, where $x_{ij}=(e_i,xe_j)$. So, a $\mathbb K$-valued function can be viewed as an infinite matrix with continuous functions as entries. Let $x\in J$, $x=(x_{ij})_{i,j\in\mathbb N}$, where each $x_{ij}\in C(X)$. 

\begin{lem}\label{lines}
The ideal $J$ consists of matrices such that the functions  $x_{ij}$ vanish at $r_i$ for any $i,j\in\mathbb N$.  

\end{lem}
\begin{proof}
As $\chi_{[0,t_i]}p_i$ is discontinuous only at $t=t_i$, the $i$-th row must vanish at this point, which guarantees continuity. 
\end{proof}

\begin{lem}
$J$ is thick.

\end{lem}
\begin{proof}
Suppose that there exists a non-zero $a\in A$ such that $ax=0$ for any $x\in J$. Suppose that $a_{ij}(t_0)\neq 0$ for some $i,j\in\mathbb N$ and for some $t_0\in X$. As $a_{ij}$ is continuous, we may assume without loss of generality that $t_0$ differs from all $t_i$, $i\in\mathbb N$. Let $e_{jk}$ denote the matrix unit in $\mathbb K$ at the intersection of $j$-th row and $k$-th column, considered as a constant function on $X$. Let $f\in C(X)$ be a function such that $f(t_0)\neq 0$, $f(t_j)=0$. Set $x(t)=e_{jj}f(t)$. Then $x\in J$ and $(ax)_{ij}(t_0)=a_{ij}(t_0)f(t_0)\neq 0$, which contradicts $ax=0$. Thus, the assumption $a\neq 0$ was wrong, and $a=0$, i.e. $J$ is thick. 
\end{proof}

\begin{lem}
$J$ is not topologically essential. 

\end{lem}
\begin{proof}
We will construct a non-zero (closed) ideal $I$ such that $J\cap I=0$. To this end, let $\xi\in H$ be a unit vector such that  its coordinates $\xi_i$ differ from zero for each $i\in\mathbb N$, and let $q\in\mathbb K$ be the projection onto $\xi$, which we consider as a constant function on $X$. Then $q_{ij}=\xi_i\xi_j$. Set $I=qA$. If $x\in I$ then it has the form $x_{ij}=\sum_{k\in\mathbb N}\xi_i\xi_ka_{kj}$ for some $a\in A$, hence $\frac{x_{ij}}{\xi_i}=\frac{x_{lj}}{\xi_l}$ for any $i,l,j\in\mathbb N$, i.e. the rows of the matrix $(x_{ij})$ are proportional. As a corollary, zeroes for different rows are the same. 

Let $y\in J\cap I$. At $t=t_k$ we have $\frac{y_{ij}(t_k)}{\xi_i}=\frac{y_{kj}(t_k)}{\xi_k}$ for any $i,j\in\mathbb N$, as $y\in I$. As $y\in J$, we have $y_{kj}(t_k)=0$ for any $k,j\in\mathbb N$. Hence, $y_{ij}(t_k)=0$ for any $i,j,k\in\mathbb N$. As $y$ is continuous and vanishes on the dense subset $\cup_{k\in\mathbb N}t_k$, it vanishes on the whole $X$.      
\end{proof}

\section{Intersection of thick ideals}

Note that the intersection of two (topologically) essential closed right ideals, $J$ and $I$, is (topologically) essential by a very simple reason: $(J\cap I)\cap K=J\cap(I\cap K)\neq 0$ for any non-zero closed right ideal $K$. By a modification of our example, we shall show that this is not the case for thick closed right ideals.


Let $J\subset A$ be the ideal constructed in the previous section. Let $b\in \mathbb K$ be the operator given by the matrix with the entries $(b_{ij})_{i,j=1}^\infty$ given by $b_{i1}=\eta_i$ and $b_{ij}=0$ for $j>1$, where 
\begin{equation}\label{norm}
\sum_{i=1}^\infty|\eta_i|^2<1/2
\end{equation} 
and $\eta_i\neq 0$ for any $i\in\mathbb N$. We use the same notation $b$ for the constant function on $X$ with the value $b$. Set $c=b+1\in C(X,\mathbb B(H))$. It follows from (\ref{norm}) that $\|b\|<1/2$, hence $c$ is invertible in $C(X,\mathbb B(H))$. Set $I=c^{-1}J=\{c^{-1}x:x\in J\}$. Clearly, $I\subset A$, and $I$ is a closed right ideal isomorphic to $J$. As $J$ is thick, $I$ is thick as well. Indeed, if $aI=0$ for some $a\in A$ then $ac^{-1}J=0$, hence $ac^{-1}=0$, which implies $a=0$.

\begin{lem}
The ideal $I\cap J\subset A$ is not thick.
\end{lem}  
\begin{proof}
Let $x\in I\cap J$. As $x\in I=c^{-1}J$, there exists $y\in J$ such that $c^{-1}y=x$. Write $x$ and $y$ as functions on $X$ taking values in infinite matrices, $x=(x_{ij})$, $y=(y_{ij})$, $i,j\in\mathbb N$, where $x_{ij},y_{ij}\in C(X)$. Recall that, by Lemma \ref{lines}, the functions $x_{kj},y_{kj}$ vanish at $t_k$ for any $k,j\in\mathbb N$.   

As $y=cx$, we have $y_{ij}=\eta_i x_{1j}+x_{ij}$, hence 
\begin{equation}\label{difference}
x_{1j}=\frac{1}{\eta_i}(y_{ij}-x_{ij}).
\end{equation}
As the right hand side in (\ref{difference}) vanishes at $t_i$ for any $j\in\mathbb N$, the left hand side also vanishes at $t_i$. But the only continuous function $x_{1j}$ that vanishes at each $t_i$, $i\in\mathbb N$, must equal zero. Thus, the first line of the matrix that represents an arbitrary element $x\in I\cap J$ is zero. 

Let $z_{11}=1$, $z_{ij}=0$ for any $(i,j)\neq(1,1)$, $z=(z_{ij})_{i,j\in\mathbb N}$. Then $z\in A$. Clearly, $zx=0$ for any $x\in I\cap J$, i.e. $I\cap J$ is not thick. 
\end{proof}

\section{Case of Hilbert $C^*$-modules and their submodules}

The definitions of essential and thick right ideals are standardly generalized to right submodules by replacing (closed) right ideals by (closed) right submodules in modules over a $C^*$-algebra $A$. Let $N$ be a right Hilbert $C^*$-module over a $C^*$-algebra $A$, and let $M\subset N$ be its closed sibmodule. 

Functionals on Hilbert $C^*$-modules give one more way to express the idea that $M$ is large in $N$. Here are three properties that characterize largeness of a submodule in a Hilbert $C^*$-module. 

\begin{itemize} 
\item[$(m1)$]
$M$ is {\it essential} in $N$ if $M\cap K\neq 0$ for any (not neccessarily closed) non-zero submodule $K\subset N$.   
\item[$(m3)$]
$M$ is {\it thick} in $N$ if $M^\perp:=\{n\in N:\langle n,m\rangle=0\ \ \forall m\in M\}=0$. 
\item[$(m5)$]
$M$ is {\it functionally thick} in $N$ if $f|_M=0$ implies $f=0$ for any functional $f$ on $N$. 
\end{itemize} 

In the case when $N=A$ and $M$ is a closed right ideal, $(m1)$ and $(m3)$ coincide with $(i1)$ and $(i3)$, respectively. It is also possible to generalize $(i2)$ and $(i4)$ to submodules in an obvious way, but we do not pursue this here.   

\begin{thm}
One has $(m1)\Rightarrow (m5)\Rightarrow (m3)$, and all three properties are different. 
\end{thm}
\begin{proof}
For $n\in N$ set $J_n=\{a\in A:na\in M\}$. Then $J_n\subset A$ is a closed right ideal.

We begin with the following well-known statement, for which we could not find a reference.

\begin{lem}\label{mod-ideal}
If $M\subset N$ is an essential submodule, then $J_n\subset A$ is an essential right ideal for any $n\in N$.

\end{lem}  
\begin{proof}
Let us prove $(i1')$ for $J_n$. Let $a\in A$, $a\neq 0$. Consider $n'=na\in N$. If $n'=0$ then $n'b=0$ for any $b\in A$, and we may choose $b$ to satisfy $ab\neq 0$. If $n'\neq 0$ then, as $M$ is essential, there exists $b\in A$ such that $n'b\in M$, $n'b\neq 0$ (this condition, similar to $(i1')$, is equivalent to $(m1)$). In both cases, there exists $b\in A$ such that $n'b\in M$. This means that $nab\in M$, hence $ab\in J_n$. In both cases we have $ab\neq 0$. 
\end{proof}

By Proposition \ref{Implication}, if $J_n$ is essential then it is thick, i.e. satisfies $(i3)$. This allows to show the implication $(m1)\Rightarrow(m5)$. Indeed, consider a functional $f$ on $N$ such that $f|_M=0$. Take an arbitrary $n\in N$. If $a\in J_n$ then $na\in M$, hence $a^*f(n)=f(na)=0$, hence $f(n)^*a=0$. Then $f(n)^*J_n=0$,a and as $J_n$ is thick, this implies that $f(n)^*=0$, i.e. $f(n)=0$. Thus, we have shown that $(m1)\Rightarrow (m5)$.

Now let us show that $(m5)\Rightarrow(m3)$. If $(m3)$ does not hold then there exists $n\in N$ such that $n\neq 0$ and $\langle n,m\rangle=0$ for any $m\in M$, i.e. $n\perp M$. Set $f_n(x)=\langle x,n\rangle$, $x\in N$. Then $f_n$ is clearly a functional on $N$, and $f_n(n)=\langle n,n\rangle\neq 0$ as $n\neq 0$, hence $f\neq 0$ and $f|_M=0$.

Let us construct an example of a functionally thick submodule that is not essential. Let $A=C(X)$, let $N=l_2(A)$ be the standard Hilbert $C^*$-module over $A$, i.e. $N$ is the set of sequences $(a_i)_{i\in\mathbb N}$, $a_i\in A$, such that $\sum_{i\in\mathbb N}a_i^*a_i$ is convergent in $A$, and let $M\subset N$ consists of sequences $(a_i)_{i\in\mathbb N}$ with the property that $a_i(t_i)=0$.  

\medskip\noindent
{\bf Claim 1.} {\it $M$ is not essential in $N$.} 
\begin{proof}
It suffices to provide some $n\in N$ such that $na\notin M$ for any $a\neq 0$. Set $n=(\xi_i)_{i\in\mathbb N}$, where $\xi_i$ are constant functions, and $\xi_i\neq 0$ for any $i\in\mathbb N$. If $\sum_{i\in\mathbb N}|\xi_i|^2<\infty$ then $n\in N$. If $na\in M$ for some $a\in A$ then $\xi_ia(t_i)=0$, which means that $a(t_i)=0$ for any $i\in\mathbb N$. As the set $\{t_i\}_{i\in\mathbb N}$ is dense in $X$, we conclude that $a=0$.
\end{proof}

\noindent
{\bf Claim 2.} {\it $M$ is functionally thick.} 
\begin{proof}
It is known that any functional $f$ on $N=l_2(A)$ is given by a sequence $f=(f_i)_{i\in\mathbb N}$, $f_i\in A$, such that the partial sums $\sum_{i=1}^k|f_i|^2$ are uniformly bounded. If $f\neq 0$ then there exists $i\in\mathbb N$ and $t_0\in X$ such that $f_i(t_0)\neq 0$. As before, we may assume that the point $t_0$ differs from the points $t_j$, $j\in\mathbb N$. Let $a_i\in A$ satisfy $a_i(t_0)\neq 0$, and let $a_j=0$ for any $j\neq i$. Then $m=(a_i)_{i\in\mathbb N}\in M$. Suppose that $f|_M=0$. Then $0=f(m)=\bar f_ia_i$, but at $t_0$ this equality fails. 
\end{proof}

An example of a thick submodule that is not functionally thick was constructed in \cite{Kaad-Skeide}. 
\end{proof}

For any Hilbert $C^*$-module $M$ there is a natural supermodule, the second dual Hilbert $C^*$-module $M''$. It follows from the arument preceeding Theorem 2.4 in \cite{Paschke} that $(M'')'=M'$, i.e. $M$ and $M''$ have the same sets of functionals. Therefore, any functional on $M''$ vanishing on $M$ is zero, hence $M$ is functionally thick in $M''$ for any Hilbert $C^*$-module $M$. We do not know whether it is essential, but in all known examples it is.

\section{(Non)transitivity of large Hilbert $C^*$-submodules}

We have already seen that if two submodules, $M$ and $K$, in a Hilbert $C^*$-module satisfy $(m1)$ then their intersection $M\cap K$ also satisfies the same property, but for $(m3)$ this need not be true. Here we address the question about transitivity of $(m1)$, $(m3)$ and $(m5)$: for $K\subset M\subset N$, if $K$ satisfies one of these conditions in $M$, and $M$ satisfies the same condition in $N$, does $K$ satisfy the same property in $N$?

\begin{prop}
Let $K$ satisfy $(m1)$ (resp., $(m5)$) in $M$, and let $M$ satisfy $(m1)$ (resp., $(m5)$) in $N$. Then $K$ satisfies $(m1)$ (resp., $(m5)$) in $N$. 

\end{prop}
\begin{proof}
For $(m1)$ this is standard: let $L\subset N$, then $(K\cap M)\cap L=K\cap(M\cap L)$, and as $0\neq P=M\cap L\subset M$, $K\cap P\neq 0$. For $(m5)$ suppose that there exists a functional $f\neq 0$ on $N$ such that $f|_K=0$. Consider its restriction $f|_M$ onto $M$. If $f|_M=0$ then we have a contradiction with $(m5)$ for $M\subset N$. If $f|_M\neq 0$ then we have a contradiction with $(m5)$ for $K\subset M$.
\end{proof}

\begin{thm}
There exists a Hilbert $C^*$-module $N$ over $A=C[0,1]$ and its submodules $K\subset M\subset N$ such that $K$ is thick in $M$, $M$ is thick in $N$, but $K$ is not thick in $N$. 

\end{thm}
\begin{proof}
Let $M=l_2(A)$ be the standard Hilbert $C^*$-module over $A$. Note that for any $m\in M$, the formula $f_m(x)=\langle x,m\rangle$ defines a functional $f_m=j(m)$ on $M$. Let $f$ be a functional on $M$. If there exists $m\in M$ such that $f_m=f$ then we may identify $f$ with $m$. In \cite{Kaad-Skeide}, and then in \cite{M-MN}, a functional $f$ on $M$ was constructed such that if $fa=j(m)$ for some $m\in M$ then $a=0$. We assume that $\|f\|<1$. Also without loss of generality we may assume that there exists $m'\in M$ such that $f(m')=1$. Indeed, both $f$ and $m'$ are sequences of elements of $A$, $f=(f_i)$, $m'=(m'_i)$, $i\in\mathbb N$, with certain convergence properties, and we may change $f_1$ and $m'_1$ to provide $f(m')=1$. 

Set $N=M\oplus A$, and write $m$ for $(m,0)$ and $n$ for $(0,1)$. Then any element of $N$ can be written as $m+na$, $m\in M$, $a\in A$. Define a sesquilinear form on $N$ by
$$
\langle m+na,m'+na'\rangle_N:=\langle m,m'\rangle+a^*f(m')^*+f(m)a'+a^*a',   
$$ 
$m,m'\in M$, $a,a'\in A$.
It follows from the Cauchy inequality for functionals, $f(m)f(m)^*\leq\|f\|^2\langle m,m\rangle$, that 
\begin{eqnarray*}
a^*f(m)^*+f(m)a&\geq& -(a^*a+f(m)f(m)^*)\\
&\geq& -(a^*a+\|f\|^2\langle m,m\rangle), 
\end{eqnarray*}
therefore 
\begin{eqnarray}\label{eqKoshi}
\langle m+na,m+na\rangle_N&=&\langle m,m\rangle+a^*f(m)^*+f(m)a+a^*a\nonumber\\
&\geq&(1-\|f\|^2)\langle m,m\rangle\ \geq \  0.
\end{eqnarray}
If $\langle m+na,m+na\rangle_N=0$ then $\langle m,m\rangle=0$, which implies also $a^*a=0$, hence $m+na=0$. Thus, $N$ is a pre-Hilbert $C^*$-module. 

\begin{lem}
$N$ is complete with respect to the norm defined by the inner product $\langle\cdot,\cdot\rangle_N$.

\end{lem}
\begin{proof}
Let $(m_i+na_i)_{i\in\mathbb N}$ be a Cauchy sequence. Then, by (\ref{eqKoshi}),
$$
(1-\|f\|^2)\langle m_i-m_j,m_i-m_j\rangle\leq \langle m_i+na_i-m_j-na_j,m_i+na_i-m_j-na_j\rangle_N
$$
vanishes as $i,j\to\infty$, hence $(m_i)_{i\in\mathbb N}$ is a Cauchy sequence in $M$. Then $(a_i)_{i\in\mathbb N}$ is a Caushy sequence in $A$.
\end{proof}

Clearly, $M$ is thick in $N$. Indeed, suppose that there exists $n'\in N$ such that $\langle n',m\rangle_N=0$ for any $m\in M$. Then $n'=m'+na$ for some $m'\in M$ and $a\in A$, and $f_{n'}=f_{m'}+fa$ is a functional on $M$. We have 
$$
0=\langle n',m\rangle_N=\langle m'+na,m\rangle_N=\langle m',m\rangle+a^*f(m)=(f_{m'}+fa)(m),
$$  
hence $f'(m)=0$ for any $m\in M$, where $f'=f_{m'}+fa$. Then $f'=0$, hence $n'=0$. 

Define the submodule $K\subset M$ by setting $K=\{m\in M:f(m)=0\}$. 

\begin{lem}
$K$ is thick in $M$.

\end{lem}
\begin{proof}
Note that $f(m')=1$ implies that $M=m'A\oplus K$ (the sum is a direct sum, but non-orthogonal one). Indeed, let $m\in M$. Write $a=f(m)\in A$. Then $f(m-m'a)=a-f(m')a=0$, hence $m-m'a\in K$, and $m=m'a+(m-m'a)$. This decomposition is unique: if $m=m'b+k$ for some $b\in A$, $k\in K$, then $m'(b-a)\in K$, i.e. $f(m')(b-a)=b-a=0$.

Suppose that there exists $m\in M$ such that $m\perp K$. Let $b=f_m(m')=\langle m',m\rangle\in A$. Then $(fb-f_m)(m')=f(m')b-b=0$. The assumption $m\perp K$ implies that $(fb-f_m)(k)=0$ for any $k\in K$. Then $fb-f_m=0$ on the whole $M=m'A\oplus K$, i.e. $fb=f_m=j(m)$. This implies $b=0$, hence $m=0$. 
\end{proof}

For any $k\in K$ we have $\langle k,n\rangle_N=f(k)=0$, hence $n\perp K$, and $K$ is not thick in $N$. 
\end{proof}

\section{Right ideals of compact operators on Hilbert $C^*$-modules related to a submodule}

 Let $\mathbb K(N)$ denote the $C^*$-algebra of compact operators on the Hilbert $C^*$-module $N$. Recall that, given $x,y\in N$, the elementary operator $\theta_{x,y}$ is defined by $\theta_{x,y}(z)=x\langle y,z\rangle$, and $\mathbb K(N)$ is the norm closure of the span of elementary operators in the Banach space of all bounded maps from $N$ to $N$. 

For a closed submodule $M\subset N$ set $J_M=\{T\in\mathbb K(N):\Ran T\subset M\}$. Clearly, $J_M$ is a closed right ideal in $\mathbb K(N)$. In this section we show that $(m1)$ (resp., $(m3)$) for $M$ in $N$ is equivalent to $(i1)$ (resp., $(i3)$) for $J_M$ in $\mathbb K(N)$.

\begin{thm}Let $M\subset N$ be a closed submodule. \\
1. The following conditions are equivalent:
\begin{itemize}
\item
$M$ is essential in $N$;
\item 
$J_M$ is essential in $\mathbb K(N)$.
\end{itemize}
2. The following conditions are equivalent:
\begin{itemize}
\item
$M$ is thick in $N$;
\item 
$J_M$ is thick in $\mathbb K(N)$.
\end{itemize}

\end{thm}
\begin{proof}
{\bf 1.} Let $M$ be essential in $N$. Let $S\in\mathbb K(N)$ be a non-zero compact operator. Then $\Ran S\subset N$ is a not necessarily closed submodule in $N$, and by essentiality, $\Ran S\cap M\neq 0$, hence there exists $m\in M$ such that $m'=S(m)\in M$ and $S(m)\neq 0$. Consider the product $ST$, where $T=\theta_{m,m'}\in\mathbb K(N)$. As $S\theta_{m,m'}=\theta_{S(m),m'}=\theta_{m',m'}$, we have 
$$
ST(x)=\theta_{m',m'}(x)=m'\langle m',x\rangle\in M
$$ 
for any $x\in N$, hence $ST\in J_M$. To see that it is non-zero, one may take $x=m'$, then $ST(m')=m'\langle m',m'\rangle$ and 
$$
\langle ST(m'),ST(m')\rangle=\langle m',m'\rangle^3\neq 0. 
$$
Therefore, $J_M$ is essential in $\mathbb K(N)$.

In the opposite direction, suppose that $J_M$ is essential in $\mathbb K(N)$. Then for any non-zero $S\in\mathbb K(N)$ there exists $T\in\mathbb K(N)$ such that $ST\in J_M$ and $ST\neq 0$. Take $n\in\mathbb N$, $n\neq 0$, $S=\theta_{n,n}$, and let 
\begin{equation}\label{eq2}
\theta_{n,n}T\in J_M,\quad \theta_{n,n}T\neq 0. 
\end{equation}
Note that the first condition in (\ref{eq2}) means that $n\langle n,T(x)\rangle\in M$ for any $x\in N$, and the second condition in (\ref{eq2}) means that there exists $x\in N$ such that $n\langle n,T(x)\rangle\neq 0$. Set $a=\langle n,T(x)\rangle\in A$. Then the conditions in (\ref{eq2}) can be written as $na\in M$ and $na\neq 0$, which imply essentiality of $M$ in $N$.

{\bf 2.} Let $M$ be thick in $N$, and suppose that $J_M$ is not thick. The latter means that there exists a non-zero $S\in\mathbb K(N)$ such that $ST=0$ for any $T\in\mathbb K(N)$. Note that $J_M$ contains operators of the form $\theta_{m,m}$ for $m\in M$,  defined by $\theta_{m,m}(x)=m\langle m,x\rangle$, hence $S\theta_{m,m}=0$ for any $m\in M$. Also note that for each $m\in M$ and any $\varepsilon>0$ there exist $x\in N$ such that $\|m-\theta_{m,m}(x)\|<\varepsilon$. Indeed, one may consider the $C^*$-subalgebra generated by $\langle m,m\rangle$ in $A$ and take $x=ma$ with $a$ from an approximate unit for this $C^*$-subalgebra. Suppose that $S(m)\neq 0$. Then
$$
S(m)=S(m)-S\theta_{m,m}(x)=S(m-\theta_{m,m}(x)),
$$
hence
$$
\|S(m)\|\leq\|S\|\|m-\theta_{m,m}(x)\|<\|S\|\varepsilon.
$$
As $\varepsilon$ was arbitrary, we have $S(m)=0$ for any $m\in M$. 

As $S\neq 0$, there exists $n\in N$ such that $S(n)\neq 0$. Set $n'=S^*S(n)$. Then $\langle n',n\rangle=\langle S(n),S(n)\rangle\neq 0$, hence $n'\neq 0$. For $m\in M$ we have 
$$
\langle n',m\rangle=\langle S^*S(n),m\rangle=\langle S(n),S(m)\rangle=0, 
$$
hence $n'\perp M$ --- contradiction with thickness of $M$.

In the opposite direction, let $J_M$ be thick, and suppose that $M$ is not thick in $N$, i.e. that there exists $n\in N$ such that $n\neq 0$ and $n\perp M$. Set $S=\theta_{n,n}\in\mathbb K(N)$. Then $S\neq 0$. Let $T\in J_M$. Then $ST(x)=n\langle n,T(x)\rangle=0$ for any $x\in N$, i.e. $SJ_M=0$ for a non-zero $S$, which contradicts thickness of $J_M$.
\end{proof}

\section{Extending the inner product by essential right ideals}

Let $A$ be a $C^*$-algebra, and let $\{J_\gamma\}_{\gamma\in\Gamma}$ be the set of all essential closed right ideals of $A$. Set
$$
M_1=\{f\in M':\exists \gamma\in\Gamma \mbox{\ such\ that\ }fx\in M \ \forall x\in J_\gamma\}.
$$
Clearly, $M\subset M_1\subset M'$. It was shown in \cite{M-MN} that $M_1$ may differ both from $M$ and from $M'$.

\begin{lem}
$M_1$ is a linear subspace of $M'$.

\end{lem}
\begin{proof}
Let $f,g\in M_1$. If $fJ_\gamma\subset M$ and $gJ_\delta\subset M$ then $(f+g)(J_\gamma\cap J_\delta)\subset M$. As $J_\gamma\cap J_\delta$ is essential, $f+g\in M_1$.
\end{proof}

Let $f\in M'$, $g\in M_1$, and let $gJ_\gamma\subset M$, $\gamma\in\Gamma$. 
For a closed right ideal $J\subset J_\gamma$ and for $x\in J$, $y\in A$ set 
$$
B^{f,g}(x,y)=(fy)(gx)\in A. 
$$
If $y\in J$ and if $fJ\subset M$ then we have $B^{f,g}(x,y)=\langle gx,fy\rangle$.

If $J=J_\alpha$ for some $\alpha\in\Gamma$ then we write this sesquilinear form as $B^{f,g}_{\alpha}$. Let $\mathcal B(J_\alpha)$ denote the Banach space of all sesquilinear and $A$-sesquilinear maps from $J_\alpha\times A$ to $A$. Then $B_\alpha^{f,g}\in\mathcal B(J_\alpha)$. The map $a\mapsto x^*ay$  defines the map $\iota_\alpha:A\to\mathcal B(J_\alpha)$.   
As $J_\alpha$ is essential, it is thick, hence the map $\iota_\alpha$ is injective.   
If $\alpha,\beta\in\Gamma$ and $J\alpha\subset J_\beta$ then $\iota_\alpha=\iota_\beta\circ r_{\alpha\beta}$, where $r_{\alpha\beta}:\mathcal B(J_\beta)\to\mathcal B(J_\alpha)$ is the restriction map.

\begin{lem}\label{independ}
Let $f,g\in M_1$, $fJ_\gamma\subset M$, $fJ_{\gamma'}\subset M$, $gJ_\delta\subset M$, $gJ_{\delta'}\subset M$, and let $B^{f,g}_{\gamma\delta}$ and $B^{f,g}_{\gamma'\delta'}$ lie in the range of $\iota_{\gamma\delta}$ and $\iota_{\gamma'\delta'}$ respectively, then there exists a unique $a\in A$ such that $\iota_{\gamma\delta}(a)=B^{f,g}_{\gamma\delta}$ and $\iota_{\gamma'\delta'}(a)=B^{f,g}_ {\gamma'\delta'}$. 

\end{lem}
\begin{proof}
Note that $J=J_{\gamma\delta}\cap J_{\gamma'\delta'}$ is also essential. As $J\subset J_{\gamma\delta}$ and $J\subset J_{\gamma'\delta'}$, there exist restrictions $r:\mathcal B(J_{\gamma\delta})\to \mathcal B(J)$ and $r':\mathcal B(J_{\gamma'\delta'})\to \mathcal B(J)$, and $r(B^{f,g}_{\gamma\delta})=r'(B^{f,g}_{\gamma'\delta'})$.

Let $a,a'\in A$ satisfy 
$$
\iota_{\gamma\delta}(a)=B^{f,g}_{\gamma\delta},\quad \iota_{\gamma'\delta'}(a')=B^{f,g}_{\gamma'\delta'}.
$$ 
Then $r(\iota_{\gamma\delta}(a))=r'(\iota_{\gamma'\delta'}(a'))$. As $\iota=r\circ\iota_{\gamma\delta}=r'\circ\iota_{\gamma'\delta'}:A\to \mathcal B(J)$ is injective, $a=a'$.
\end{proof}

Set 
$$
M_0=\{f\in M_1:B^{f,g}\in \iota(A)\ \ \forall g\in M_1 \}, 
$$
and let $\widetilde M=\overline M_0$ be the norm closure of $M_0$ in $M'$.
By Lemma \ref{independ}, this does not depend on the choice of essential right ideals for $f$ and $g$.
If $J$ is an essential closed right ideal in $A$ such that $fJ\subset M$, $gJ\subset M$, and if $B^{f,g}=\iota(a)$, $a\in A$,  then $B^{f,g}(x,y)=x^*ay$.

\begin{lem}
$\widetilde M$ is an $A$-module. 

\end{lem}
\begin{proof}
It is clear that $\widetilde M$ is a linear space. Let us show that it is closed under multiplication by elements of $A$. If $g\in M_1$ and $B^{f,g}=\iota(b)$ for some $b\in A$ then $B^{f,g}(x,y)=x^*by$, $x\in J$, $y\in A$. Let $a\in A$, then 
$$
B^{fa,g}(x,y)=(fay)(gx)=f(gx)ay=B^{f,g}(x,ay)=x^*b(ay)=x^*(ba)y, 
$$
hence 
\begin{equation}\label{B}
B^{fa,g}=\iota(ba).
\end{equation} 
\end{proof}

For $f,g\in M_0$, set $\langle g,f\rangle=a$, where $\iota(a)=B^{f,g}$. 
\begin{lem}
$\langle\cdot,\cdot\rangle$ is an inner product on the pre-Hilbert $C^*$-module $M_0$.

\end{lem}
\begin{proof}
Clearly, $\langle g,f_1+\lambda f_2\rangle=\langle g,f_1\rangle+\lambda\langle g,f_2\rangle$, when $g,f_1,f_2\in M_0$, $\lambda\in\mathbb C$. 

Let $B^{f,g}=\iota(b)$, $a\in A$. Then $\langle g,f\rangle=b$. By (\ref{B}), $\langle g,fa\rangle=ba$, thus the inner product is $A$-linear with respect to the second argument.   

Let $B^{f,g}=\iota(b)$, $B^{g,f}=\iota(c)$, let $fJ,gJ\subset M$, let $J$ be essential, and let $x,y\in J$. Then 
\begin{equation}\label{B1}
\langle gx,fy\rangle=B^{f,g}(x,y)=x^*by=\iota(b)(x,y),
\end{equation}
\begin{equation}\label{C1}
\langle fy,gx\rangle=B^{g,f}(y,x)=y^*cx=\iota(c)(y,x). 
\end{equation}
It follows from (\ref{B1}) and (\ref{C1}) that 
$$
(\iota(b)(x,y))^*=\iota(c)(y,x).
$$
But
$$
\iota(b^*)(x,y)=x^*b^*y=(y^*bx)^*=(\iota(b)(y,x))^*,
$$
hence $\iota(b^*)(x,y)=\iota(c)(x,y)$, i.e. $x^*b^*y=x^*cy$ for any $x,y\in J$. As $J$ is essential, we have $b^*=c$,
hence $\langle f,g\rangle^*=\langle g,f\rangle$.

In particular, this implies that $c=\langle f,f\rangle$ is selfadjoint for any $f\in M_0$. Let $fJ\subset M$, $J$ essential. Then $x^*cx=\langle fx,fx\rangle$ is positive for any $x\in J$. 

\begin{lem}\label{L}
Let $c\in A$ be selfadjoint, let $J$ be an essential closed right ideal in $A$, and let $x^*cx\geq 0$ for any $x\in J$. Then $c\geq 0$.
\end{lem}
\begin{proof}
Suppose that $c$ is not positive. Then $c=c_+-c_-$ with positive $c_+,c_-$ and $c_+c_-=0$. Suppose that $c_-\neq 0$. Then the right ideal $I=c_-A$ is non-zero. As $J$ is essential, there exists $a\in A$ such that $x=c_-a\in I\cap J$, $x\neq 0$. Then $x^*(c_+-c_-)x=a^*c_-^3a\leq 0$, hence $a^*c_-^3a=0$. Then $c_-^{3/2}a=0$, hence, multiplying by $c_-^{1/2}$, we get $c_-^2a=0$. Then $c_-^2aa^*=0$. Suppose that $c_-aa^*=0$. Then $c_-aa^*c_-=0$, hence $c_-a=0$ --- a contradiction to $x\neq 0$, hence $c_-aa^*\neq 0$. Set $b=aa^*\geq 0$. Thus we have 
\begin{equation}\label{c}
c_-b\neq 0 \quad \mbox{and}\quad c_-^2b=0. 
\end{equation}
The two conditions in (\ref{c}) contradict each other. Indeed, if $c_-^2b=0$ then $(c_-bc_-)^2=(c_-bc_-)(c_-bc_-)=0$, hence, by uniqueness of the positive square root, $c_-bc_-=0$, which implies that $c_-b^{1/2}=0$. Multiplying this by $b^{1/2}$, we get $c_-b=0$.
Thus the assumption $c_-\neq 0$ is wrong. 
\end{proof}

It follows from Lemma \ref{L} that $c=\langle f,f\rangle\geq 0$ for any $f\in M_0$.

Now suppose that $\langle f,f\rangle=0$ for some $f\in M_0$, and let $m\in M$. Let $fJ\subset M$, $x\in J$. Then $\langle m,fx\rangle=f(m)x$. It follows from the Cauchy inequality 
$$
\langle fx,m\rangle\langle m,fx\rangle\leq \|m\|^2\langle fx,fx\rangle=0
$$
that $f(m)x=0$ for any $x\in J$. If $J$ is essential then this implies $f(m)=0$. As $m\in M$ is arbitrary, this means that $f=0$.

\end{proof}

\begin{lem}\label{L2}
Let $c\in A$, $c\geq 0$, let $J$ be an essential closed right ideal in $A$, and let $\sup_{x\in J,\|x\|=1}\|x^*ax\|=1$. Then $\|a\|=1$.

\end{lem}
\begin{proof}
Clearly, $\|x^*ax\|\leq \|a\|\cdot\|x\|^2$, hence $\|a\|\geq 1$, so the problem is to show the opposite inequality.
Let $\{u_\lambda\}_{\lambda\in\Lambda}$ be a left approximate unit for $J$ (\cite{Blackadar}, II.5.3.3), which means that it is an approximate unit for the $C^*$-subalgebra $J\cap J^*$ and $\lim_\Lambda u_\lambda x-x=0$ for any $x\in J$. 

It follows from the assumption that $x^*ax\leq 1$ for any $x\in J$ with $\|x\|=1$. In particular, this holds for $x=u_\lambda$: $u_\lambda a u_\lambda\leq 1$. Multiplying this by $x$, we get $x^*u_\lambda a u_\lambda x\leq x^*x$. Passing here to the limit, we obtain that $x^*ax\leq x^*x$.

To prove that $\|a\|\leq 1$, consider first the case when $A$ is unital. Then $1-a$ is selfadjoint and $x^*(1-a)x=x^*x-x^*ax\geq 0$ for any $x\in J$. Then, by Lemma \ref{L}, $1-a$ is positive, hence $\|a\|\leq 1$. 

For the non-unital case we have to pass to the unitalization $A^+$ of $A$ and to notice that $J$ is essential in $A^+$ as well. Indeed, let $I$ be a non-zero right ideal in $A^+$. Suppose that $I\cap A=0$ in $A^+$. Then $I$ is the set of all scalars, which is not a right ideal, hence $I'=I\cap A\subset A$ is a non-zero ideal in $A^+$, hence in $A$. Then $I'\cap J$ is non-zero by essentiality of $J$ in $A$. It remains to apply Lemma \ref{L} in $A^+$.   
\end{proof}

For $f\in M_0$, let $\|f\|_{M'}=\sup_{m\in M,\|m\|\leq 1}\|f(m)\|$ be the norm of $f$ in $M'$, and let $\|f\|^2=\|\langle f,f\rangle\|$.

\begin{thm}\label{completion}
$\|f\|_{M'}=\|f\|$ for any $f\in M_0$.

\end{thm}
\begin{proof}
Let $c=\langle f,f\rangle$. Then 
$$
\|fx\|^2=\|\langle fx,fx\rangle\|=\|x^*cx\|\leq \|x\|^2\|c\|.
$$
By Lemma \ref{L2},
\begin{eqnarray*}
\|c\|^2&=&\sup_{x\in J,\|x\|=1}\|x^*c^2x\|\\
&=&\sup_{x\in J,\|x\|=1}\|x^*c\|^2\\
&=&\sup_{x\in J,\|x\|=1}\|x^*\langle f,f\rangle\|^2\\
&=&\sup_{x\in J,\|x\|=1}\|\langle fx,f\rangle\|^2\\
&=&\sup_{x\in J,\|x\|=1}\|f(fx)\|^2\\
&\leq& \|f\|_{M'}^2\sup_{x\in J,\|x\|=1}\|fx\|^2 \\
&\leq&  
\|f\|_{M'}^2\|c\|,
\end{eqnarray*}
whence $\|f\|^2=\|c\|\leq \|f\|_{M'}^2$.

On the other hand, as $M\subset M_0$,
\begin{eqnarray*}
\|f\|&=&\sup_{g\in M_0,\|g\|\leq 1}\|\langle g,f\rangle\|\geq \sup_{m\in M,\|m\|\leq 1}\|\langle m,f\rangle\|\\
&=&
\sup_{m\in M,\|m\|\leq 1}\|f(m)\|\ =\ \|f\|_{M'}.
\end{eqnarray*}
\end{proof}

\begin{thm}
$\widetilde M\subset M'$ is a Hilbert $C^*$-module over $A$ with the inner product $\langle\cdot,\cdot\rangle$ that extends the inner product on $M$. This extension of the inner product is unique. 

\end{thm}
\begin{proof}
 By Lemma \ref{completion}, the completion of the pre-Hilbert $C^*$-module $M_0$ with respect to the norm $\|\cdot\|$ coinsides with its closure in $M'$. 

To prove uniqueness, consider two inner products $\langle\cdot,\cdot\rangle_i$, $i=1,2$, on $\widetilde M$ that coincide on $M$. Let $g,f\in M_0$, $\langle g,f\rangle_i=a_i$. Let $J\subset A$ be an essential closed right ideal such that $fJ,gJ\subset M$. Then $x^*a_1y=\langle gx,fy\rangle=x^*a_2y$ for any $x,y\in J$, i.e. $x^*(a_2-a_1)y=0$. As $J$ is thick, $x^*a_2-x^*a_1=0$ for any $x\in J$. Passing to adjoints, we get $(a_2-a_1)x=0$ for any $x\in J$, whence $a_2=a_1$. 
\end{proof}

\section{Examples}

\begin{example}
Let $\mathbb K(H)\subset A\subset\mathbb B(H)$ be an algebra of bounded operators on a Hilbert space $H$. The two-sided ideal $J=\mathbb K(H)$ of compact operators is essential in $A$. Let $M=l_2(A)$. Then $f=(f_i)_{i\in\mathbb N}\in M'$ if the sums $\sum_{i=1}^j f_i^*f_i$, $j\in\mathbb N$, are uniformly bounded. Then the series $\sum_{i\in\mathbb N}f_i^*f_ix$ is norm convergent for any $x\in J$, hence $fJ\subset M$ for any $f\in M'$, hence $M_1=M'$. For general $A$ we cannot say anything about $M_0$, but when $A=\mathbb B(H)$ then the series $\sum_{i\in\mathbb N}f_i^*f_i$ strongly converges to an element of $A$, hence determines an element of $A$ as a multiplier, and the same holds for the series $\sum_{i\in\mathbb N}f_i^*g_i$ for $g\in M'$. Thus, $M_0=\widetilde M=M'$. 

\end{example}

\begin{example}
Let $A=C(X)$ for a compact Hausdorff space $X$, and let $M=l_2(A)$. An ideal $J\subset C(X)$ is essential if there exists a nowhere dense closed subset $F\subset X$ such that $J=\{f\in C(X):f|_F=0\}$ (see e.g. \cite{Azarpanah}). 
Let $f\in M'$, $f=(f_i)_{i\in\mathbb N}$. Then $f\in M_1$ if the series $\sum_{i\in\mathbb N}|f_i|^2$ converges pointwise on $X\setminus F$, for some nowhere dense subset $F\subset X$, to a continuous function on $X\setminus F$, and $f\in M_0$ if this series converges pointwise on $X\setminus F$ to a function that is continuous on the whole $X$. For example, consider $X=[0,1]$, and define $f_n$ to be the piecewise linear function such that $f_n(x)=0$ for $x\in [0,\frac{1}{n+1}]\cup[\frac{1}{n-1},1]$ and $f_i(\frac{1}{i})=1$. Set $f=(f_i)_{i\in\mathbb N}$, $g=(f_{2i})_{i\in\mathbb N}$. Then $f,g\in M'$, and $f\in \widetilde M$, while $g\notin\widetilde M$. Note that for $A=C[0,1]$ the second dual module $M''$ coincides with $M$ \cite{FMT}, so the inner product extends much further than $M''$.    

\end{example}

\end{document}